\documentclass[11pt]{amsart}
\usepackage[english]{babel}
\usepackage{graphicx}
\usepackage{framed}
\usepackage[normalem]{ulem}
\usepackage{amsmath}
\usepackage{amsthm}
\usepackage{amssymb}
\usepackage{amsfonts}
\usepackage{enumerate}
\usepackage[utf8]{inputenc}
\usepackage[top=1 in,bottom=1in, left=1 in, right=1 in]{geometry}

\theoremstyle{definition}
\newtheorem{theorem}{Theorem}[section]
\newtheorem{conj}[theorem]{Conjecture}
\newtheorem{lemma}[theorem]{Lemma}
\newtheorem{remark}[theorem]{Remark}

\theoremstyle{definition}

\theoremstyle{remark}
\newtheorem{conv}{Convention}

\def\irr#1{{\rm  Irr}(#1)}
\def\irri#1#2{{\rm Irr}(#1 | #2)}
\def\cd#1{{\rm  cd}(#1)}
\def\Aut#1{{\rm Aut}(#1)}
\def\Out#1{{\rm Out}(#1)}

\newcommand{\SU}{{\rm SU}}
\newcommand{\GU}{{\rm GU}}
\newcommand{\GL}{{\rm GL}}

\newcommand{\SL}{{\rm SL}}

\newcommand{\bG}{\ensuremath{\mathbf{G}}}
\newcommand{\bT}{\ensuremath{\mathbf{T}}}

\newcommand{\PSL}{\mathrm{PSL}}

\newcommand{\wt}[1]{\widetilde{#1}}

\newcommand{\ind}{\mathrm{Ind}}
\newcommand{\res}{\mathrm{Res}}

\title[Automorphisms and characters of groups of type A]{Action of automorphisms on irreducible characters of groups of type \textsf{A}}
\author{Farrokh Shirjian and Ali Iranmanesh}

\address{Farrokh Shirjian: Department of Pure Mathematics, Faculty of Mathematical Sciences, Tarbiat Modares University, Tehran, Iran}
\email{farrokh.shirjian@modares.ac.ir}
\email{fashirjian@gmail.com}
\address{Ali Iranmanesh: Department of Pure Mathematics, Faculty of Mathematical Sciences, Tarbiat Modares University, P.O. Box 14115-137, Tehran, Iran}
\email{iranmanesh@modares.ac.ir}

\thanks{Corresponding Author: Ali Iranmanesh.}
\subjclass[2010]{Primary 20C33, 20C15} \keywords {Finite reductive groups, Automorphisms, McKay conjecture, Generalized Gelfand-Graev representations.}

\date{Aug 2021}

\begin{document}

\maketitle
\begin{abstract}
Let $G$ be a finite group isomorphic to $\SL_n(q)$ or $\SU_n(q)$ for some prime power $q$. In this paper, we give an explicit description of the action of automorphisms of $G$ on the set of its irreducible complex characters. This is done by showing that irreducible constituents of  restrictions of irreducible characters of $\GL_n(q)$ (resp. $\GU_n(q)$) to $\SL_n(q)$ (resp. $\SU_n(q)$) can be distinguished by the rational classes of their unipotent support which are equivariant under the action of automorphisms. Meanwhile, we give a criterion to explicitly determine whether an irreducible character is the constituent of a given generalized Gelfand-Graev character of $G$. As an application, we give a short proof of the global side of Sp{\" a}th's criterion for the inductive McKay condition for the irreducible characters of $G$. 
 \end{abstract}


\section{Introduction}
A finite reductive group is the fixed-point subgroup $G:=\bG^F$ of a connected reductive group $\bG$ defined over the finite field $\mathbb{F}_q$ of characteristic $p > 0$, where $F : \bG \rightarrow \bG$ is the Frobenius map corresponding to this $\mathbb{F}_q$-structure. In recent years, many conjectures in representation theory of finite reductive groups have been ``reduced" to check some new technical conditions about quasi-simple groups of Lie type. These new conditions to check involve analysing the action of automorphisms of a quasi-simple Lie-type group on the set of its irreducible characters.\\

\noindent\textbf{Question.} (\cite[Problem 2.33]{MK}) For $G$ a quasi-simple group of Lie type, determine the action of $\Aut G$ on $\irr G$.\\

For $\epsilon \in \{\pm\}$, we set $\SL^-_n(q)=\SU_n(q)$ and $\SL^+_n(q)=\SL_n(q)$. In this paper, we answer the above question for $G=\SL^\epsilon_n(q)$. Indeed, we generalize an earlier result of Brunat and Himstedt \cite{bru, bh} concerning the action of automorphisms on the set of semisimple and regular characters.

It has been shown by Lusztig that the irreducible characters of $\bG^F$ can be partitioned into the so-called geometric Lusztig series, labelled by the semisimple $\bG^*$-classes of ${\bG^*}^{F^*}$, where $(\bG^* , F^*)$ denotes a pair dual to $(\bG, F)$. If such a series is labelled by a semisimple class with representative $s$, then it contains $|{A_{\bG^*}(s)}^{F^*}|$ semisimple characters, where $A_{\bG^*} (s) = C_{\bG^*} (s)/ C^\circ_{\bG^*} (s)$ is the component group of $s$. The set of semisimple (and regular) characters of $G=\bG^F$ can be naturally parametrized by pairs $(s, \xi)$ where $s$ runs over a set of representatives of the semisimple classes of ${\bG^*}^{F^*}$ and $\xi \in \irr{A_{G^*}(s)}$, where $A_{G^*}(s):={A_{\bG^*}(s)}^{F^*}$. Using the usual Gelfand-Graev characters of $G$, Brunat and Himstedt turned the action of automorphisms on regular characters to the action on the corresponding labels $(s, \xi)$, see \cite{bru, bh}.

In this paper, we use the Kawanaka construction of generalized Gelfand-Graev characters (GGGCs) to separate the irreducible characters in geometric Lusztig series and obtain in this way a labelling system for arbitrary irreducible characters of $\SL^\epsilon_n(q)$ in terms of three parameters (see Theorem \ref{main1} and Section \ref{sec6}). Indeed, it is shown that the various components of restrictions of irreducible characters of $\GL^\epsilon_n(q)$ to $\SL^\epsilon_n(q)$ can be distinguished by the $\SL^\epsilon_n(q)$-classes of their unipotent support. Using this labelling system we then study the action of outer automorphisms on irreducible characters by considering the induced action on the corresponding parameters.

To state our result, we need to introduce some more notation. For any $\lambda \in \irr {C^\circ_{G^*}(s)}$, we denote by $A_{G^*}(s)_\lambda$ the stabilizer of $\lambda$ under $A_{G^*}(s)$. Moreover, the outer automorphism group $\Out G$ is well known to be generated by diagonal, field, and graph automorphisms. In the case of diagonal automorphisms, the action on $\irr G$ is well understood by work of Lusztig \cite{lu88}. In the sequel, we write $F_p$ for a generator of the group of field automorphisms, and write $\gamma$ for a generator of the group of graph automorphisms of $G$.
The main goal of this paper is to show the following.\\

\textbf{Main Theorem (Theorem \ref{main-result}).} Assume that $G=\SL^\epsilon_n(q)$.  For any semisimple element $s \in G^*$ and any unipotent character $\lambda \in \irr {C^\circ_{G^*}(s)}$, there exists a morphism $\omega^0_{s, \lambda}:H^1(F, Z(\bG)) \rightarrow  \irr {A_{G^*}(s)_\lambda}$ such that for the irreducible character $\chi_{s, \lambda, \omega^0_{s, \lambda}(z)} \in \irr G$ parametrized by triple $(s, \lambda, \omega^0_{s, \lambda}(z))$ for some $z \in H^1(F, Z(\bG))$, one has
\[ {}^\sigma \chi_{s, \lambda, \omega^0_{s, \lambda}(z)}=\chi_{{\sigma^*}^{-1}(s), \sigma^*(\lambda), \omega^0_{s, \lambda}(\sigma(z))},  \]
where $\sigma \in \langle F_p, \gamma \rangle$ and $\sigma^* \in \Aut {G^*}$ is its dual automorphism.\\

Consequently, using the Main Theorem, we obtain a short proof of the global side of the so-called inductive McKay condition, cf. \cite[Theorem 2.12]{sp12}, for the irreducible characters of $G=\SL^\epsilon_n(q)$, see also \cite[Theorem 4.1]{CS}. In what follows, let $\wt G=\GL^\epsilon_n(q)$ where $\GL^{-}_n(q)=\GU_n(q)$.\\

\textbf{Corollary (Theorem 7.2).}
If $\wt \chi=\chi_{\wt s, \lambda} \in \irr {\wt G}$, then for the irreducible character $\chi_0=\chi_{s, \lambda, 1} \in \irri {G} {\wt \chi}$ we have
\[ (\wt G \times \langle F_p, \gamma \rangle)_{\chi_0}={\wt G}_{\chi_0} \times (\langle F_p, \gamma \rangle)_{\chi_0}.  \]

\begin{remark}
We should remark that, based on a parametrisation of irreducible characters of $G$ given by Bonnaf{\' e} in \cite{Bo00}, Cabanes and Sp{\" a}th recently established an equivariant Jordan decomposition for the irreducible characters of groups of type \textsf{A} in \cite{CS}. In this paper, we used an $\Aut {\wt G}$-equivariant parametrisation of irreducible characters of $\wt G=\GL^\epsilon_n(q)$ obtained in \cite[Theorem 3.1]{CS13} to show that the various components of restrictions of $\chi_{\wt s, \lambda}$ to $\SL^\epsilon_n(q)$ can be distinguished by the rational classes of their unipotent support, and obtain in this way a parametrisation of $\irr G$. It is not clear whether the two parametrisations of $\irr G$ (used here and in \cite{Bo00}) coinsides. Moreover, our proof of the Main Theorem mainly follows the ideas developed by Brunat and Himstedt in \cite{bru, bh} which is different from the approach used in \cite{CS}. However, the relation of the two results will be discussed in Section \ref{eq-Sec}.
\end{remark}

The rest of the paper is organized as follows. In Section 2, we introduce some basic facts and properties of generalized Gelfand-Graev characters. Section 3 is devoted to the earlier Lusztig and Brunat-Himstedt results on determining the action of automorphisms on the set of regular and semisimple characters. In Section 4, we give an interpretation of Lusztig's result describing the action of diagonal automorphisms on the irreducible characters and then obtain a parametrisation of $\irr {\SL^\epsilon_n(q)}$. Sections 5 and 6 are at the heart of this paper. In section 5, we give a criterion to explicitly determine whether an irreducible character is the constituent of a given GGGC of $\SL^\epsilon_n(q)$ (see Theorem \ref{main1}), which might be of an independent interest. Note that inner automorphisms act trivially on irreducible characters, and the action of diagonal automorphisms was already described in Section 4. Therefore, we describe in Section 6 the action of field and graph automorphisms on the set of irreducible characters of $\SL^\epsilon_n(q)$ (see Theorem \ref{main-result}). Some potential applications of the Main Theorem into other character theoretical problems are discussed in Section 7. Finally, we discuss the relation of our results to \cite{CS} in Section \ref{eq-Sec}.

\subsection{Notation.}
We denote by $\res^G_H \chi$ the restriction of a character $\chi$ of $G$ to some subgroup $H \leq G$. Also the induction of a character $\psi$ of $H$ to $G$ is denoted by $\ind_H^G \psi$. For $N \lhd G$ and $\chi \in \irr G$, we denote by $\irri N \chi$ the set of irreducible constituents of the restriction $\res^G_N \chi$. Finally, the stabilizer of $\psi \in \irr N$ under the action of $G$ on $\irr N$ is denoted by $G_\psi$. Other notation are standard or will be defined where needed.

\section{Generalized Gelfand-Graev Characters}

Let $\bG$ be a connected reductive algebraic group defined over an algebraic closure $\mathbb{K} = \overline{\mathbb{F}_p}$ of the finite field of prime order $p$ and let $F : \bG \to \bG$ be a Frobenius endomorphism defining an $\mathbb{F}_q$-rational structure $G = \bG^F$ on $\bG$. Assuming $p$ is a good prime for $\bG$ a theory of generalised Gelfand--Graev characters (GGGCs) was developed by Kawanaka in \cite{ka}. These are certain characters $\Gamma_u$ of $G$ which are defined for any unipotent element $u \in G$. Note that $\Gamma_u = \Gamma_v$ whenever $u, v \in G$ are $G$-conjugate so the GGGCs are naturally indexed by the unipotent conjugacy classes of $G$.

\subsection{Unipotent Supports and Wave Front Sets}

From this point forward, we assume that $p$ is a good prime for $\bG$. Let $\rho \in \irr G$ be an irreducible character and $\mathcal{O}$ an $F$-stable unipotent conjugacy class of $\bG$. We will denote by $AV(\rho,\mathcal{O})$ the average value $\sum_{g \in \mathcal{O}^F} \rho(g)$ of $\rho$ on the rational points $\mathcal{O}^F$. We say $\mathcal{O}$ is a \emph{unipotent support} of $\rho$ if $\mathcal{O}$ has the maximum dimension for the property that $AV(\rho,\mathcal{O}) \neq 0$. Given a unipotent element $v \in \bG$, we denote by $\mathcal{O}_v$ the $\bG$-conjugacy class containing $v$. With this, we say that $\mathcal{O}$ is a \emph{wave front set} of $\rho$ if:
\begin{enumerate}
	\item $\langle \Gamma_u, \rho \rangle \neq 0$ for some $u \in \mathcal{O}$ and
	\item $\langle \Gamma_v, \rho \rangle \neq 0$ implies $\dim \mathcal{O}_v \leqslant \dim\mathcal{O}$ with $v \in G$ any unipotent element.
\end{enumerate}

Geck \cite[Theorem 1.4]{Ge} and Taylor \cite[Theorems 13.8 and 14.10]{Tay16} have shown that whenever $p$ is good for $\bG$, any irreducible character $\rho$ of $G$ has a unique unipotent support $\mathcal{O}_{\rho}$ and a unique wave front set $\mathcal{O}_{\rho}^*$. These turn out to be dual in the following sense. Let $\rho^* \in \irr G$ be the unique irreducible character such that $\rho^* = \pm D_\bG(\rho)$ where $D_\bG(\rho)$ is the Alvis--Curtis dual of $\rho$. Then $\mathcal{O}_{\rho^*} = \mathcal{O}_{\rho}^*$, see \cite[Lemma 14.15]{Tay16}. In other words, the unipotent support of the Alvis--Curtis dual of $\rho$ is the wave front set of $\rho$.

\subsection{Restrictions of GGGCs}
From this point forward, we assume that $\bG=\SL_n(\mathbb{K})$, $\wt \bG=\GL_n(\mathbb{K})$, $F_q: \wt \bG \rightarrow \wt \bG$ is the standard Frobenius map given by $(a_{ij})\mapsto (a^q_{ij})$ where $q$ is a $p$-power, and $\gamma_0: \wt \bG \rightarrow \wt \bG$ be the automorphism of $\wt \bG$ defined by $(a_{ij})\mapsto {{(a_{ij})}^{tr}}^{-1}$ where ${(a_{ij})}^{tr}$
is the transpose of $(a_{ij})$. We also write $(\bG^*, F^*)$ and $(\wt \bG^*, F^*)$ for pairs dual to $(\bG, F)$ and $(\wt \bG, F)$, respectively. Using the notation $\GL^-_n(q)=\GU_n(q)$ and $\SL^-_n(q)=\SU_n(q)$, we set $\epsilon \in \{\pm\}$ so that $\wt G=\wt \bG^F=\GL^\epsilon_n(q)$ and $G=\bG^F=\SL^\epsilon_n(q)$ where $F={\gamma_0}^{\frac{1-\epsilon 1}{2}}  \circ F_q$.

Note that all primes $p$ are good for $\bG$. Let $i: \bG \hookrightarrow \wt {\bG} $ be a regular embedding which is compatible with $F$. Note that $i$ restricts to an isomorphism on the variety of unipotent elements. For clarity, given a unipotent element $u\in \wt G$, we will denote by $\wt {\Gamma_u}$ (resp. $\Gamma_u$) the corresponding generalized Gelfand-Graev character of $\wt G$ (resp. $G$).  We also denote by $A_\bG(u)$ the component group of the centralizers $C_\bG(u)/C^\circ_\bG(u)$, and by $\mathcal{Z}(\bG)$ the component group $Z(\bG)/Z(\bG)^\circ$. By the construction of the GGGC it is deduced that
\begin{equation}\label{G1}
\wt {\Gamma_u}= \ind_G^{\wt G} {\Gamma_u},
\end{equation}

and all GGGCs of $\wt G$ are obtained in this way, cf. \cite[(2.1)]{Ge93}. The following lemma is a direct consequence of some well known results. However we include a proof for comprehensiveness.

\begin{lemma}\label{restrict-GGGCs}
For any unipotent element $u \in \wt G$, the pairwise distinct GGGCs of $G$ appeared with non-zero multiplicity in $\res^{\wt G}_G ~{\wt {\Gamma_u}}$ are parametrized by $\Gamma_{u_a}$, where $u_a$ denotes a representative of the $G$-class of $u$ indexed by $a \in H^1(F, A_\bG(u))$.

\end{lemma}
\begin{proof}
First note that since $C_{\wt {\bG}}(u)$ is connected, then $\mathcal{O}^F_u$ intersects $\wt G$ in a unique $\wt G$-conjugacy class and hence $\wt {\Gamma_v}=\wt {\Gamma_u}$ for any $v \in \mathcal{O}^F_u$. Since $\bG \lhd \wt \bG$, using the Mackey formula we get the following isomorphism of functors
\[ \res^{\wt G}_{G}\circ \ind_{G}^{\wt G} \cong \bigoplus_g ad(g), \]
where $g$ runs over the coset representatives of $\wt G/G$. In particular, we have
\[ \res^{\wt G}_G ~{\wt {\Gamma_u}}= \bigoplus_g \Gamma_{g^{-1}ug},\]

since $ad(g)(\Gamma_u)=\Gamma_{g^{-1}ug}$ by \cite[Proposition 2.2]{Ge93}.  Since $G$-conjugate unipotent elements possess the same GGGC, we deduce that the distinct GGGCs of $G$ appeared in the above decomposition are parametrized by the $G$-conjugacy classes of $u$. Also, it is well-known that the $G$-classes inside $\mathcal{O}^F_u$ are parametrized by the $F$-conjugacy classes of $A_\bG(u)$. Hence we are done.
\end{proof}

When there is no confusion, we may identify $\Gamma_{u_{a}}$ with $\Gamma_{a}$ to simplify the notation.

\begin{remark}\label{remark-gelf}
To illustrate Lemma \ref{restrict-GGGCs} it is useful to consider the special case when $u$ is a regular unipotent element of $\wt G$. It is well known that all regular unipotent elements form a single $\wt G$-conjugacy class and hence a unique GGGC $\Gamma:=\wt {\Gamma_u}$. Then, Lemma \ref{restrict-GGGCs} confirms the well known decomposition $\res^{\wt G}_G~ {\Gamma}=\bigoplus_z \Gamma_{z}$ where $z$ runs over $H^1(F, \mathcal{Z}(\bG))$, and $\Gamma_z$ are the usual Gelfand-Graev characters of $G$. This is because $A_\bG(u) \cong \mathcal{Z}(\bG)$ by \cite[14.24]{dm91}.
\end{remark}


\section{Action of automorphisms on regular and semisimple characters}

In this section, we consider the action of automorphisms on semisimple and regular characters of $G$.

\subsection{Automorphisms of $\bG^F$}
Let $F_p: \wt \bG \rightarrow \wt \bG$ be the Frobenius endomorphism given by
\[  F_p:(a_{ij}) \mapsto (a_{ij}^p).          \]
Note that $F_p$ defines an $\mathbb{F}_p$-structure on $\wt \bG$ and $\bG$, and it is considered as a generator of the group of field automorphisms of $G$. Let $\gamma : \wt \bG \rightarrow \wt \bG$ be an automorphism of $\wt \bG$ and $\bG$ induced by the graph automorphism of the corresponding Dynkin diagram. By \cite[Theorem 2.5.1]{GLS}, we then have
\[  (\wt G: \langle F_p, \gamma \rangle)/Z(\wt G) \cong \Aut G.    \]

A key idea to determine the action is to relate the action of outer automorphisms on irreducible characters to the action of the outer automorphisms on the corresponding GGGCs, which is easily described by the following lemma.

\begin{lemma}\cite[Proposition 11.10]{Tay18}\label{GGGC-action}
For any unipotent element $u \in G$ and any bijective morphism $\sigma: \bG \rightarrow \bG$ which commutes with $F$, we have ${}^\sigma \Gamma_u=\Gamma_{\sigma(u)}$.
\end{lemma}

\begin{conv}\label{conv1}
Note that any GGGC depends on the actual choice of the corresponding rational unipotent element. Let $u$ be a unipotent element of $\wt G$. An easy computation using the Jordan normal form shows that $\langle F_p, \gamma \rangle$ acts trivially on unipotent classes of $\wt \bG$. So $\wt {\Gamma}_u$ is $\langle F_p, \gamma \rangle$-stable. Let $\mathcal{A}=\{u_a | a \in H^1(F, A_\bG(u)) \}$ denote the set of representatives of $G$-classes in $\mathcal{O}_u^F$.  By \cite[Proposition 4.2]{CS} and Lemma \ref{GGGC-action}, there exists some member of $\mathcal{A}$ whose corresponding GGGC is $\langle F_p, \gamma \rangle$-stable. Without loss of generality, we call this element $u_1$.

\end{conv}

\begin{center}
{\fbox{From this point forward we suppose that Convention \ref{conv1} holds.}}
\end{center}

\subsection{Regular and semisimple characters}

The regular characters, by definition, are the irreducible constituents of usual Gelfand-Graev characters. Indeed, for the unique Gelfand-Graev character $\Gamma$ of $\wt G$ we have
\[  \Gamma=\sum_{\wt s} \chi_{\wt s},         \]
where $\wt s$ runs over the $\wt G^*$-classes of semisimple elements and $\chi_{\wt s}$ denotes the regular character of $\wt G$ corresponding to $\wt G^*$-class of $\wt s$. By \cite[(8.4)]{bo}, for any semisimple element $s \in G^*$, there is a surjective group homomorphism
\[ {\hat{\omega}}_s^0: H^1(F, \mathcal{Z}(\bG)) \rightarrow \irr {A_{G^*}(s)}.  \]

By Remark \ref{remark-gelf}, the usual Gelfand-Graev characters of $G$ are labelled by $\Gamma_z$ where $z \in H^1(F, \mathcal{Z}(\bG))$. Moreover, by \cite[14.49]{dm91}, the multiplicity free characters $\Gamma_z$ and $\res^{\wt G}_G \chi_{\wt s}$ have exactly one common irreducible constituent. We write ${\chi_{s,1}}$ for the common constituent of $\Gamma_1$ and $\res^{\wt G}_G ~\chi_{\wt s}$, and we set $\rho_{s, 1}=\epsilon_\bG \epsilon_{C^\circ_{\bG^*}(s)} D_\bG(\chi_{s, 1})$ where $\epsilon_\bG=(-1)^{{\rm{rk}}_{\mathbb{F}_q}(\bG)}$.

For $z \in H^1(F, Z (\bG))$, we denote by $g_z$ an element of $\bG$ such that $g^{-1}_z F(g_z)$ belongs to $Z(\bG)$ and represents $z$. Then the inner automorphism $\mathrm{int}~g_z$ belongs to $\Aut {\bG, F}$. We denote by $\tau^\bG_z$ its image in $\Out G$.
\begin{lemma}\cite[Proposition 6.3]{bo}\label{tau}
The morphism $\tau^\bG: H^1(F, Z (\bG)) \rightarrow \Out G$; $z \mapsto \tau^\bG_z$ is an injective group homomorphism.
\end{lemma}

Note that, by Lemma \ref{tau}, if $g$ and $h$ are two elements of $\bG$ such that $g^{-1}F(g)$ and $h^{-1}F(h)$ represent the same $F$-class of $Z(\bG)$, then ${g}$ and $h$ act identically on $G$ by conjugation. The following compact theorem determines the action of $\Out G$ on the set of regular and semisimple elements.
\begin{theorem}
For any semisimple element $s \in G^*$ and any $\xi \in \irr {A_{G^*}(s)}$, we define
\[ \chi_{s, \xi}={}^{g_z} \chi_{s, 1}~~~ {\rm{and}} ~~~~\rho_{s, \xi}={}^{g_z} \rho_{s,1},      \]
where $z \in H^1(F, \mathcal{Z}(\bG))$ such that ${\hat{\omega}}_s^0(z)=\xi$, and $g_z \in \bG$ such that $g_z^{-1}F(g_z) \in z$ (as explained above). Then

\begin{enumerate}

\item The set
\[ \{\chi_{s, \xi} | s \in G^* ~{\rm{semisimple}}, \xi \in \irr {A_{G^*}(s)} \} \] consists of regular characters of $G$, and the characters $\rho_{s, \xi}$ are the semisimple characters of $G$.

\item For any semisimple element $s \in G^*$, $z \in H^1(F, \mathcal{Z}(\bG))$, and $\sigma \in \langle \gamma, F_p \rangle$ we have
\[ {}^{\sigma} \chi_{s, {\hat{\omega}}_s^0(z)}=\chi_{{\sigma^*}^{-1}(s), {\hat{\omega}}_s^0(\sigma(z))},\]
and
\[ {}^{\sigma} \rho_{s, {\hat{\omega}}_s^0(z)}=\rho_{{\sigma^*}^{-1}(s), {\hat{\omega}}_s^0(\sigma(z))}.\]

\item The character $\chi_{s, \xi}$ is an irreducible constituent of $\Gamma_z$ if and only if ${\hat{\omega}}_s^0(z)=\xi$.

\end{enumerate}
\end{theorem}
\begin{proof}
Parts 1 and 3 are Corollaire 15.14 of \cite{bo}. The second part follows using the similar argument as in \cite[Theorem 3.5]{bh} and \cite[Proposition 3.3]{bru}.
\end{proof}



\section{Action of diagonal automorphisms}\label{sec4}
In this section, we give a reinterpretation of the Lusztig result \cite[5.1]{lu88} on the action of diagonal automorphisms on arbitrary irreducible characters in a unified way that is consistent with those results of the previous section (see also \cite[15.14]{CE}). In the sequel, we will use the Jordan decomposition of characters introduced in \cite[Theorem 3.1]{CS13} to parametrize the irreducible characters of $\wt G$. In particular, the parameters of the irreducible characters of $\wt G$ are chosen so that they are equivariant with respect to the action of automorphisms. 

For $\wt s \in \wt G^*$, let $s=\wt s Z(\wt \bG^*)$, $A_{\bG^*}(s)$ the component group of centralizer of $s$ in $\bG^*$, and $A_{G^*}(s)=(A_{\bG^*}(s))^{F^*}$. By \cite[Theorem 3.1]{CS13}, the irreducible characters of $\wt G$ are parametrized $(\wt s, \lambda) \mapsto \chi_{\wt s, \lambda}$ where $\wt s \in  \wt G^*$ is a semisimple element and $\lambda \in \mathcal{E}(C_{\wt G^*}(\wt s),1)=\mathcal{E}(C^\circ_{G^*}(s),1)$. Then the restriction of $\chi_{\wt s, \lambda}$ to $G$ has $|A_{G^*}(s)_\lambda|$ irreducible constituents where $A_{G^*}(s)_\lambda$ denotes the stabilizer of $\lambda$ under $A_{G^*}(s)$. Note that by \cite[11.6]{bo}, the labels of irreducible characters of $\wt G$ can be chosen such that $\res^{\wt G}_{G} ~\chi_{\wt s, \lambda}=\res^{\wt G}_{G} ~\chi_{z\wt s, \lambda}$ for any $z \in ({\rm{Ker}}~ i^*)^{F^*}$, thus the labels for characters as above are well defined.

Here we give a parametrization of the action of diagonal automorphisms on constituents of the restriction to $G$ of $\chi_{\wt s, \lambda}$. Set

\[ \res^{\wt G}_{G} ~\chi_{\wt s, \lambda} = \sum_{\xi \in \irr {{A_{G^*}(s)}_{\lambda}}}  \chi_{s, \lambda, \xi}.  \]

For any $\lambda \in \mathcal{E}(C_{\wt G^*}(s),1)$, the dual of the natural inclusion map $i^*_\lambda: A_{G^*}(s)_\lambda \rightarrow A_{G^*}(s)$ leads to a surjective homomorphism
\[  \hat{i}^*_\lambda: \irr {A_{G^*}(s)} \rightarrow \irr {A_{G^*}(s)_\lambda}.                          \]

Let $z \in H^1(F, \mathcal{Z}(\bG))$ and $\tau^{\bG}_z \in \Out G$ be the corresponding inner automorphism which acts on $G$ by conjugation by $g_z$, see Lemma \ref{tau}. Note that, as a diagonal automorphism, $\tau^{\bG}_z$ permutes the irreducible characters $\chi_{s, \lambda, \xi}$ where $\xi$ runs over $\irr {A_{G^*}(s)_\lambda}$. Thus we have

\[  {}^{g_z} \chi_{s, \lambda, \xi}= \chi_{s, \lambda, \xi {\hat{\omega}}_{s, \lambda}^0(z)},  \tag{4.1.1} \label{Eq3.1.1}\]

where ${\hat{\omega}}_{s, \lambda}^0:= \hat{i}^*_\lambda \circ {\hat{\omega}}_{s}^0:H^1(F, \mathcal{Z}(\bG)) \rightarrow  \irr {A_{G^*}(s)_\lambda}$ is a surjective homomorphism.

\section{Distribution of irreducible characters as constituents of GGGCs}\label{sec5}
Any irreducible character $\chi \in \irr {\SL_n(q)}$ occur as an irreducible constituent of restriction to $G$ of some $\wt \chi=\rho_{\wt s, \lambda} \in \irr {\wt G}$. As explained in Section 2, one can associate to $\wt \chi$ a wave front set $\mathcal{O}^*_{\wt \chi}$ , and for any rational unipotent element $u \in {\mathcal{O}^*_{\wt \chi}}^F$, there is a corresponding GGGC $\wt {\Gamma_u}$. In general, it is known that the multiplicity of $\wt \chi$ in $\wt {\Gamma_u}$ is small. In case $\wt G=\GL^\epsilon_n(q)$, the multiplicity has been proved to be exactly one.

\begin{lemma}\cite[Corollary 15.7]{Tay16}\label{GGGC-GL-1}
For any $\rho \in \irr {\wt G}$ and any unipotent element $u \in \wt G$, we have

\begin{equation*}
\langle \wt {\Gamma_u}, \rho \rangle_{\wt G} =
\begin{cases}
\displaystyle{1} & \text{if~} u \in \mathcal{O}^*_\rho \\ \vspace{.2cm}
\displaystyle{0} & \text{if~} \mathcal{O}_u \nsubseteq \overline{\mathcal{O}^*_\rho},
\end{cases}
\end{equation*}
where $\overline{\mathcal{O}^*_\rho}$ is the Zariski closure of ${\mathcal{O}^*_\rho}$.
\end{lemma}

Thus we get the following consequence.

\begin{lemma}\label{GGGC-GL-2}
Any $\rho \in \irr {\wt G}$ is a constituent of the GGGC $\wt {\Gamma}_{{\mathcal{O}^*_\rho}^F}$.
\end{lemma}
\begin{proof}
By Lemma \ref{GGGC-GL-1}, $\rho$ is a constituent of GGGCs corresponding to $\wt G$-classes inside ${\mathcal{O}^*_\rho}^F$. Since $C_{\wt \bG}(u)$ is connected for any $u \in {\mathcal{O}^*_\rho}$, we conclude that ${\mathcal{O}^*_\rho}^F$ is a single $\wt G$-conjugacy class.
\end{proof}

Recall that we are assuming Convention \ref{conv1} holds.

\begin{lemma}\label{GGGC-SL}
For any $\rho \in \irr {\wt G}$ and any unipotent element $u \in G \cap {\mathcal{O}^*_\rho}^F$, we have
\[ \langle {\Gamma_u}, \res^{\wt G}_G ~\rho \rangle_{G}=1.\]
\end{lemma}
\begin{proof}
Using Lemma \ref{GGGC-GL-1} and Frobenius reciprocity we have
\begin{align*}
\langle {\Gamma_u}, \res^{\wt G}_G ~\rho \rangle_{G}&=\langle \ind_G^{\wt G}~{\Gamma_u}, \rho \rangle_{\wt G}=\langle {\wt {\Gamma_u}}, \rho \rangle_{\wt G}=1.
\end{align*}
\end{proof}

Let $\wt \chi=\chi_{\wt s, \lambda} \in \irr {\wt G}$. By Lemmas \ref{GGGC-GL-1} and \ref{GGGC-GL-2}, $\wt \chi$ is a constituent of $\wt {\Gamma_u}$ where $u$ is a representative of the unique $\wt G$-conjugacy class inside ${\mathcal{O}^*_{\wt \chi}}^F$. Restricting to $G$ and using Lemmas \ref{G1} and \ref{GGGC-SL}, we conclude that the multiplicity free characters $\Gamma_{u_a}$, with $a \in H^1(F, A_\bG(u))$, and $\res^{\wt G}_G~ {\chi_{\wt s, \lambda}}$ have a unique common irreducible constituent. We write ${\chi_{s, \lambda, 1}}$ for the common constituent of $\Gamma_{u_1}$ and $\res^{\wt G}_G~ \chi_{\wt s, \lambda}$. Therefore, if $u_1 \in {\mathcal{O}^*_{\chi_{\wt s, \lambda}}}$, then $\langle \Gamma_{u_1},  {\chi_{s, \lambda, 1}}\rangle=1$.

\subsection{Action of automorphisms and GGGCs}
In order to relate the action of outer automorphisms on arbitrary irreducible characters of $G$ to their action on the GGGCs, we first need to determine the distribution of the irreducible characters as constituents of some suitable GGGC. This is the main aim of this section.

\begin{lemma}\label{phiu}
For $u=u_1$ there exists a surjective homomorphism $\varphi_u:H^1(F, Z(\bG)) \rightarrow H^1(F, A_\bG(u))$ such that
\[ {}^{g_z} \Gamma_{a}=\Gamma_{{\varphi_u(z)a}},\]
for any representative $g_z \in \bG$ of corresponds to $z \in H^1(F, \mathcal{Z}(\bG))$. 
\end{lemma}

\begin{proof}
Recall that for $z \in H^1(F, Z (\bG))$, we denote by $g_z$ an element of $\bG$ such that $g^{-1}_z F(g_z)$ belongs to $Z(\bG)$ and represents $z$. Moreover, we have ${}^{g_z} \Gamma_u=\Gamma_{{g_z}^{-1}ug_z}$. So we may represent $H^1(F, \mathcal{Z}(\bG))$  by the $F$-classes of elements $g_z^{-1}F(g_z)$ where $z \in H^1(F, Z(\bG))$.

Recall that we have a regular embedding $i : \bG \rightarrow \wt{\bG}$ and a canonical quotient map $C_{\wt \bG} (\wt u) \rightarrow A_{\wt \bG} (\wt u)$ where $\wt u \in \wt \bG$ such that $i(u)=\wt u$. The group $\wt \bG$ is an almost direct product $\wt \bG = \bG Z(\wt{\bG})$ so we have an almost direct product $C_{\wt \bG} (\wt u)= C_\bG(u)Z(\wt{\bG} )$. The restriction of $i$ to the centraliser $C_\bG(u)$ induces a surjective
map $A_\bG(u) \rightarrow A_{\wt \bG} (\wt u)$. Combining this with the natural inclusion maps $Z(\bG) \rightarrow C_\bG(u) \rightarrow A_\bG(u)$ we obtain the exact sequence
\[  Z(\bG) \rightarrow A_\bG(u) \rightarrow A_{\wt \bG} (\wt u) \rightarrow 1,\]

cf. \cite[p. 78]{Tay12}. We thus obtain a canonical surjective homomorphism $-: Z(\bG) \rightarrow A_\bG(u)$, since $A_{\wt \bG} (\wt u)=1$ for $\wt \bG=\GL_n(k)$. From this, we get the surjective morphism on the corresponding $F$-classes
\[ \varphi_u: H^1(F, Z(\bG)) \rightarrow H^1(F, A_\bG(u)) \]
\[ z \mapsto \lfloor\overline{g_z^{-1}F(g_z)}\rfloor,\]
where $\lfloor x\rfloor$ denotes the $F$-class of $A_\bG(u)$ represented by $x \in A_\bG(u)$. Note that $\varphi_u$ is a homomorphism since $g_z^{-1}F(g_z) \in Z(\bG)$ and the morphism $-: Z(\bG) \rightarrow A_\bG(u)$ is surjective. 

Recall that each $G$-classes of $u$ is indexed by a corresponding element $a \in H^1(F, A_\bG(u))$. Indeed, by \cite[\S 4.3.4]{Gebook}, the $G$-class of $g_z^{-1}ug_z$ corresponds to $\lfloor\overline{g_z^{-1}F(g_z)}\rfloor \in H^1(F, A_\bG(u))$. Therefore, by setting $u_1=u$, we have
\[ {}^{g_z} \Gamma_{u_1}=\Gamma_{g_z^{-1}ug_z}=\Gamma_{\lfloor\overline{g_z^{-1}F(g_z)}\rfloor}=\Gamma_{\varphi_u(z)},\]
and hence, assuming $u_a=g^{-1}_{z'}ug_{z'}$ for some $z' \in \bG$, we have
\[ {}^{g_z} \Gamma_a={}^{g_z} \Gamma_{u_a}= \Gamma_{g^{-1}_zg^{-1}_{z'}ug_{z'}g_z}=
\Gamma_{\lfloor \overline{{(g_{z'}g_z)}^{-1}F(g_{z'}g_z)\rfloor}}=
\Gamma_{\lfloor \overline{g^{-1}_z F(g_z)} \overline{g^{-1}_{z'} F(g_{z'})} \rfloor}=
\Gamma_{\varphi_u(z) \varphi_u(z')},\]
since $g^{-1}_zF(g_z) \in Z(\bG)$ and $a=\lfloor \overline{g^{-1}_{z'}F(g_{z'})}  \rfloor=\varphi_u(z') \in H^1(F, A_\bG(u))$.
\end{proof}

We are now in a position to state the main result of this section which concerns the distribution of an irreducible character as a constituent of some GGGCs. This can be seen as a generalized and analogue form of \cite[Proposition 15.13 and Corollaire 15.14 ]{bo} (see also \cite[Theorem 3.1]{bru}).

\begin{theorem}\label{main1}

For any $\chi_{\wt s, \lambda} \in \irr {\wt G}$ and any $u \in {\mathcal{O}^*_{\wt \chi}}^F$, we have $\langle {\Gamma_{u_a}}=\Gamma_a, \res^{\wt G}_G ~\chi_{\wt s, \lambda} \rangle_{G}=1$ where $a \in H^1(F, A_\bG(u))$. We write $\chi_{s, \lambda, 1}$ for the common constituent of $\Gamma_1$ and $\res^{\wt G}_G~\chi_{\wt s, \lambda}$. Moreover, for $\xi \in \irr {A_{G^*}(s)_\lambda}$, we define
\[\chi_{s, \lambda, \xi}={}^{g_z} \chi_{s, \lambda, 1}, \]
where $z$ is any element of $H^1(F, \mathcal{Z}(\bG))$ such that ${\hat{\omega}}_{s, \lambda}^0(z)=\xi$, and $g_z \in \bG$ such that $g_z^{-1}F(g_z) \in z$. Finally, for any $a \in H^1(F, A_\bG(u))$ and $\xi \in \irr {A_{G^*}(s)_\lambda}$, the character $\chi_{s, \lambda, \xi}$ is a constituent of $\Gamma_a$ if and only if ${\hat{\omega}}_{s, \lambda}^0 \circ \varphi_u^{-1}(a)=\xi$.
\end{theorem}
\begin{proof}
The first part is a well-known result of Lusztig covered in Section 4. For the second part, note that if $u_1 \in {{\mathcal{O}^*}^F_{\chi_{\wt s, \lambda}}} \cap G$, then we have
\[   \Gamma_{1}= {\chi_{s, \lambda, 1}} + \cdots. \label{abc} \tag{5.1.1} \]

Now considering $g_z \in \bG$ acting on both sides of \eqref{abc}, we get
\[    \langle \Gamma_{\varphi_u(z)},    {\chi_{s, \lambda, {\hat{\omega}}_{s, \lambda}^0(z)}} \rangle =1.\]
Moreover, we claim that Ker $\varphi_u \subset$ Ker ${\hat{\omega}}_{s, \lambda}^0$. Otherwise, there exists $z_0 \in \mathrm{Ker}~ \varphi_u \setminus \mathrm{Ker}~ {\hat{\omega}}_{s, \lambda}^0$. Now acting $g_{z_0}$ on \eqref{abc} we get 
\[   \langle \Gamma_{1},   {\chi_{s, \lambda, {\hat{\omega}}_{s, \lambda}^0(z_0)}}\rangle \neq 0, \]
which is in contradiction with Lemma \ref{GGGC-SL}. Therefore, we obtain the claim. In particular the map ${\hat{\omega}}_{s, \lambda}^0 \circ \varphi_u^{-1}$ is well defined.
\end{proof}
\section{An equivariant character labelling}\label{sec6}
In this section, we are going to determine the action of automorphisms on arbitrary irreducible characters of $\SL^\epsilon_n(q)$. Note that assuming Convention \ref{conv1} holds, Theorem \ref{main1} gives a parametrization of irreducible characters of $\SL^\epsilon_n(q)$ in terms of the three labels $s, \lambda, \xi$ as follows:

Let $\chi \in \irr G$. Following Section 4, we may choose an irreducible character $\wt \chi=\rho_{\wt s, \lambda} \in \irr {\wt G}$ lying above $\chi$. One can associate to $\wt \chi$ a wave front set $\mathcal{O}^*_{\wt \chi}$, and for any unipotent element $u \in {\mathcal{O}^*_{\wt \chi}}^F$,  there is a corresponding GGGC $\wt {\Gamma_u}$. Moreover, we may choose $u\in {\mathcal{O}^*_{\wt \chi}}^F$ as in Convention \ref{conv1}. Note that all the irreducible characters of $\wt G$ lying above $\chi$ have the same wave front set \footnote{For, if $\wt \chi_1, \wt \chi_2 \in \irri {\wt G} {\chi}$, then $\langle \Gamma_{u}, \chi \rangle\neq 0$ and $\langle \Gamma_{v}, \chi \rangle\neq 0$ for some $u \in {\mathcal{O}^{*F}_{\wt \chi_1}}$ and $v \in {\mathcal{O}^{*F}_{\wt \chi_2}}$. Now inducing to $\wt G$, we get $\wt {\Gamma_u}= \wt \chi_1+\wt \chi_2+\cdots$ and $\wt {\Gamma_v}= \wt \chi_1+\wt \chi_2+\cdots$. Now the uniqueness of the wave front sets of  $\wt \chi_1$ and $\wt \chi_2$ implies that $\mathcal{O}^*_{\wt \chi_1}= \mathcal{O}^*_{\wt \chi_2}$.}. Therefore, the choice of $u$ here is independent of the actual choice of $\wt \chi$. Then $\chi$ is the common constituent of $\res^{\wt G}_G~{\wt \chi}$ and some $\Gamma_{a}$ where $a=\varphi_u(z)$ for some $z \in H^1(F, \mathcal{Z}(\bG))$, (see Lemmas \ref{GGGC-SL} and \ref{phiu}). This common constituent is labelled as $\chi=\chi_{s, \lambda, \omega^0_{s, \lambda}(z)}$, (see Theorem \ref{main1}). In the sequel, for any semisimple element $s \in G^*$, let $\mathfrak{U}(C^\circ_{G^*}(s))$ denote the set of unipotent characters of $C^\circ_{ G^*}(s)$ (note that $\mathcal{E}(C^\circ_{ G^*}(s), 1)=\mathcal{E}(C_{\wt G^*}(\wt s), 1)$).

\begin{theorem}\label{main-result}
Assume that Convention \ref{conv1} holds. For any semisimple element $s \in G^*$, $\lambda \in \mathfrak{U}(C^\circ_{G^*}(s))$, and $z \in H^1(F, \mathcal{Z}(\bG))$, one has
\[ {}^\sigma \chi_{s, \lambda, \omega^0_{s, \lambda}(z)}=\chi_{{\sigma^*}^{-1}(s), \sigma^*(\lambda), \omega^0_{s, \lambda}(\sigma(z))},  \]
where $\sigma \in \langle F_p, \gamma \rangle$ and $\sigma^* \in \Aut {G^*}$ is its dual in the sense of \cite[Definition 2.1]{CS13}.
\end{theorem}
\begin{proof}

 Let $\chi_{\wt s, \lambda} \in \irr {\wt G}$ lying over $\chi_{s, \lambda, \omega^0_{s, \lambda}(z)}$ as described in Section 4. Let $u, u_1 \in \mathcal{O}^F_{\chi_{\wt s, \lambda}}$ be as in Convention \ref{conv1}. By Theorem \ref{main1}, $\chi_{s, \lambda, \omega^0_{s, \lambda}(z)}$ is the common constituent of $\Gamma_{\varphi_u(z)}$ and $\res^{\wt G}_G~\chi_{\wt s, \lambda}$. We write $A_{\wt s, \lambda}$ for the set of irreducible constituents of $\res^{\wt G}_G~\chi_{\wt s, \lambda}$, and $A_{u_z}$ for the set of constituents of $\Gamma_{\varphi_u(z)}$. So we have
\[  A_{\wt s, \lambda} \cap  A_{u_z}=\{\chi_{s, \lambda, \omega^0_{s, \lambda}(z)} \}.\]
Since $Z(\wt \bG)$ is connected and the parameters of the irreducible characters of $\wt G$ are chosen so that they are $\Aut {\wt G}$-equivariant (see Section \ref{sec4}), Theorem 3.1 of \cite{CS13} implies that
\[ {}^\sigma  \chi_{\wt s, \lambda}=\chi_{{\sigma^*}^{-1}(\wt s), \sigma^*(\lambda)}. \]
Therefore,
\[ {}^\sigma A_{\wt s, \lambda}=  A_{{\sigma^*}^{-1}(\wt s), \sigma^*(\lambda)}. \tag{6.1.1} \label{5.1.1}\]
On the other hand, by Lemma \ref{GGGC-action}, we have
\[ {}^\sigma \Gamma_{\varphi_u(z)}={}^{\sigma g_z} \Gamma_1={}^{\sigma(g_z) \sigma} \Gamma_1={}^{\sigma(g_z)} \Gamma_1 = {}^{g_{\sigma(z)}} \Gamma_1=\Gamma_{\varphi_u(\sigma(z))},\]

since $^{\sigma} \Gamma_1=\Gamma_1$, and ${\sigma(g_z)}^{-1}F(\sigma(g_z))$ and ${g^{-1}_{\sigma(z)}}F({g_{\sigma(z)}})$ belong to the same $F$-class (note that, by definition, ${g^{-1}_{z}}F({g_{z}}) \in z$. Therefore, $\sigma({g^{-1}_{z}})\sigma(F({g_{z}}))=\sigma({g^{-1}_{z}})F(\sigma(g_{z})) \in \sigma(z)$ and so, by Lemma \ref{tau}, ${g_{\sigma(z)}}$ and $\sigma({g_z})$ act identically on GGGCs). Thus we get
\[   ^{\sigma} A_{u_z}= A_{u_{\sigma(z)}}. \tag{6.1.2} \label{5.2.2}      \]

Now the relations \eqref{5.1.1} and \eqref{5.2.2} imply that
\[  \{{}^{\sigma} \chi_{s, \lambda, \omega^0_{s, \lambda}(z)}\}={}^{\sigma}A_{\wt s, \lambda} \cap ~{}^{\sigma} A_{u_z}=A_{{\sigma^*}^{-1}(\wt s), \sigma^*(\lambda)} \cap A_{u_{\sigma(z)}}= \{ \chi_{{\sigma^*}^{-1}(s), \sigma^*(\lambda), \omega^0_{s, \lambda}(\sigma(z))} \}.         \]
\end{proof}


\section{Applications}
In this section, we review some potential applications of Theorem \ref{main-result}. In \cite{CS}, among other things, the authors investigated the global side of the so-called inductive McKay conditions. Indeed, they verified the so-called stabilizer condition as follow:

\begin{theorem}\cite[Theorem 4.1]{CS}
For any $\wt \chi \in \irr {\wt G}$ there is a $\chi_0 \in \irri G {\wt \chi}$ such that
\[ (\wt G \times \langle F_p, \gamma \rangle)_{\chi_0}={\wt G}_{\chi_0} \times (\langle F_p, \gamma \rangle)_{\chi_0}.  \]
\end{theorem}

Using Theorem \ref{main-result}, we can precisely determine the character $\chi_0$.

\begin{theorem}
If $\wt \chi=\chi_{\wt s, \lambda} \in \irr {\wt G}$, then for the irreducible character $\chi_0=\chi_{s, \lambda, 1} \in \irri G {\wt \chi}$ we have
\[ (\wt G \times \langle F_p, \gamma \rangle)_{\chi_0}={\wt G}_{\chi_0} \times (\langle F_p, \gamma \rangle)_{\chi_0}.  \]
\end{theorem}
\begin{proof}
By Theorem \ref{main-result}, we have
\[ {}^\sigma \chi_{s, \lambda, 1}=\chi_{{\sigma^*}^{-1}(s), \sigma^*(\lambda), 1},  \]
for any $\sigma \in \langle F_p, \gamma \rangle$. Moreover, for any non-trivial element $g_z \in \wt G/GZ(\wt G)$ we have
\[ {}^{g_z} \chi_{s, \lambda, 1}=\chi_{s, \lambda, \omega^0_{s, \lambda}(z)}.  \]
The rest of the proof is straightforward.
\end{proof}

Theorem \ref{main-result} may also be applied to other wide range of questions and conjectures concerning the interaction between the structure of finite groups and the set of their irreducible character degrees, such as the Huppert conjecture \cite{Hu00}. In what follows, let $\cd H$ denote the set of irreducible character degrees of the finite group $H$.

\begin{conj}{(Huppert Conjecture)}
Let $S$ be a finite non-abelian simple group and $H$  be a finite group with $\cd S= \cd H$. Then $H \cong S \times A$ for some finite abelian group $A$.
\end{conj}
So far the Huppert conjecture has been verified for the families of simple alternating \cite{BTZ} and sporadic groups, and only few families of simple groups of Lie type of rank up to 3. The main approach to prove this conjecture consists of two main steps. We first reduce the conjecture to almost simple groups by showing that $H/C_H(H')$ is an almost simple group with socle $S$. The main obstacle to verify the final implication is to compare $\cd S$ with $\cd {H/C_H(H')}$ for which we need to analyze the action of $\Aut S$ on $\irr S$. Therefore, Theorem \ref{main-result} provides an explicit tool to compute and compare character degrees of different almost simple groups with socle $\PSL^\epsilon_n(q)$ and opens up the way for proving the conjecture for simple linear  and unitary groups of arbitrary large ranks.\\ 

\section{The relations}\label{eq-Sec}
In this section, we discuss how the Main Theorem of this paper might be related to the results of \cite[\S 8]{CS}. The following argument suggested by anonymous colleague who has kindly permited to include this material in the paper.

We now follow the notation of \cite[\S 8]{CS}. Fix a semisimple element $\wt s \in \wt G^*$. For each $\rho \in \irr {{W(\wt s)}^{F^*}} $ fixed by the automorphism induced by $F^*$,  Lusztig–Srinivasan \cite[Thm 3.2]{LS77} have shown that we have an irreducible character 

\[R^{\wt \bG}_\rho[\wt s]=\epsilon_{\wt \bG}\epsilon_{C_{\wt \bG^*}(\wt s)} \frac{1}{|W(\wt s)|} \sum_{w \in W(\wt s)} \hat{\rho}(wF^*) R^{\bG^*}_{{\wt \bT^*}_w}(\wt s),  \]

where $\hat{\rho} \in \irr {W(\wt s) \rtimes \langle F^* \rangle}$ is the Lusztig's canonical extension of $\rho$ defined as in \cite[Section 35]{bo}. Similarly we get a class function

\[R^{{C_{\bG^*}(\wt s)}}_\rho= \frac{1}{|W(\wt s)|} \sum_{w \in W(\wt s)} \hat{\rho}(wF^*) R^{{C_{\bG^*}(\wt s)}}_{{\wt \bT_w^*}}(1)  \]
for ${C_{\bG^*}(\wt s)}$ in which $1$ means the trivial character. The preferred extension of $\rho$ is precisely the unique extension $\hat{\rho}$ such that the corresponding $R^{{C_{\bG^*}(\wt s)}}_\rho$ is a unipotent character of ${C_{\bG^*}(\wt s)}$ , see \cite[Thm. 2.2]{LS77}.

Now the following theorem gives a slight strengthening of Lusztig--Srinivasan’s result.

\begin{theorem}
There exists a unique bijection $J^{\wt \bG}_{\wt s}: \mathcal{E}(\wt G, \wt s) \rightarrow \mathcal{E}(C_{\wt G^*}(\wt s), 1)$ such that
\begin{enumerate}
\item $\mathcal{E}(\wt G, \wt s)=\{R^{\wt \bG}_\rho[\wt s]: \rho \in \irr {{W(\wt s)}^{F^*}}  \} $,
\item $J^{\wt \bG}_{\wt s}(R^{\wt \bG}_\rho[\wt s])=R^{{C_{\bG^*}(\wt s)}}_\rho$ for any $\rho \in \irr {{W(\wt s)}}$,
\item ${}^\sigma R^{\wt \bG}_{{}^{\sigma^*}\rho}[{\sigma^*}(\wt s)]=R^{\wt \bG}_\rho[\wt s]$ for any $\rho \in \irr {{W(\wt s)}}$ and any pair of dual automorphisms $(\sigma, \sigma^*)$.
\end{enumerate}
\end{theorem}
\begin{proof}
(1) is \cite[Thm 3.2]{LS77}. This implies that the set of virtual characters $\{ R^{\wt \bG}_{\wt \bT^*}(\wt s): \wt \bT^* \leq C_{\wt \bG^*}(\wt s)~~{\rm{an~F-stable~maximal~torus}} \}$ spans the same subspace of the space of class functions as in $\mathcal{E}(\wt G, \wt s)$. Hence, there can be only one bijection satisfying (1) and so (2) holds by \cite[Thm 2.2]{LS77}. Finally (3) is shown in \cite[Prop. 8.5]{CS}.
\end{proof}

We now consider a character $\chi_{\wt s, \lambda}$ where $\lambda \in \mathcal{E}(C_{\wt G^*}(\wt s), 1)$ is a unipotent character. By the above theorem, we have $\lambda=R^{{C_{\bG^*}(\wt s)}}_\rho$ for some $\rho \in \irr {{W(\wt s)}}$ and also $\chi_{\wt s, \lambda}=R^{\wt \bG}_\rho[\wt s]$. The restriction of the character on the right hand side is
\[R^{ \bG}_\rho[ s]=\epsilon_{\bG}\epsilon_{C_{\bG^*}( s)} \frac{1}{|W^\circ(s)|} \sum_{w \in W^\circ(s)} \hat{\rho}(wF^*) R^{\bG^*}_{{ \bT_w^*}}(s),  \]
where $s \in G$ is the image of $\wt s$ under the surjective map $\wt \bG^* \rightarrow \bG^*$ dual to the regular embedding $\bG \rightarrow \wt \bG$. Fixing a unipotent element $u_1 \in {\mathcal{O}^*_{\wt s, \lambda}}^F$ in the wave front set of the class, we write $\chi_{s, \lambda, 1}=R^\bG_\rho[s]_1$ for the common irreducible constituent of the GGGC $\Gamma_{u_1}$ and $\res^{\wt G}_G (\chi_{\wt s, \lambda})=R^\bG_\rho[s]$.

It is well known that we have an isomorphism $A_{\bG^*}(s) \cong W(s)/W^\circ(s) $ that is $F^*$ -equivariant. The group $W(s)^{F^*}/{W^\circ(s)}^{F^*}$ also acts on $\irr {W(s)}^{F^*}$. The stabilizer $A_{G^*}(s)_\lambda \leq A_{G^*}(s)$ of $\lambda$ is mapped isomorphically to the stabilizer of $\rho$ in $W(s)^{F^*}/{W^\circ(s)}^{F^*}$. So we can write $A_{G^*}(s)_\rho$ for $A_{G^*}(s)_\lambda$.

Next, for any $\xi \in \irr {A_{G^*}(s)_\lambda}=\irr {A_{G^*}(s)_\rho}$, we define the character $\chi_{s, \lambda, \xi}={R^\bG_\rho[s]}_\xi$ by applying a coresponding diagonal automorphism to the chosen constituent. Note that 
\[{}^{\sigma^*} \lambda={}^{\sigma^*} R^{C_{\wt \bG^*}(\wt s)}_\rho=R^{{C_{\wt \bG^*}(\sigma^*(\wt s))}}_{{}^{\sigma^*}\rho}\]
 which follows from the same argument as above but this time we are working in $C_{\wt \bG^*}(\wt s)$ not its dual group, so there is no need to transpose the action of the automorphism, see \cite[Corollary 1.50]{Tay12}.
 
 Finally, \cite[Lem 8.9]{CS} yields that
 \[\chi_{s, \lambda, \xi}={R^\bG_\rho[s]}_\xi={}^\sigma R^{\wt \bG}_{{}^{\sigma^*}\rho}[{\sigma^*}(\wt s)]_{{}^{\sigma^*}\xi} ={}^\sigma \chi_{\sigma^*(s), \sigma^*(\lambda), \sigma^*(\xi)}, \]
which is the Main Theorem's result.

\section*{Acknowledgment}
This work was supported by the Grant No. 98012009 from the Iranian National Science Foundation (INSF). The paper was initiated in a conversation with Olivier Brunat at IMJ-PRG. We would like to thank him for pointing out this approach. The authors are deeply indebted to the anonymous colleague who suggested the materials of Section \ref{eq-Sec} and kindly allowed to include it in the paper. The first author would also like to thank Marc Cabanes, Frank Himstedt, Farideh Shafiei, and Jay Taylor for their helpful communication and comments over the course of working on this paper.

\end{document}